\documentclass[reqno]{amsart}
\usepackage{txfonts}
\usepackage{amssymb}
\usepackage{amsthm}
\usepackage{amsmath}
\usepackage{mathrsfs}

\usepackage[utf8]{inputenc}

\theoremstyle{plain}
\newtheorem{theorem}{Theorem}[section]
\newtheorem{lemma}[theorem]{Lemma}

\newtheorem{proposition}[theorem]{Proposition}

\theoremstyle{definition}
\newtheorem{definition}[theorem]{Definition}

\newtheorem{remark}[theorem]{Remark}

\newcommand{\C}{\mathbb{C}}
\newcommand{\N}{\mathbb{N}}
\newcommand{\R}{\mathbb{R}}

\newcommand{\mcC}{\mathcal{C}}
\newcommand{\mcE}{\mathcal{E}}
\newcommand{\mcG}{\mathcal{G}}
\newcommand{\mcH}{\mathcal{H}}
\newcommand{\mcL}{\mathcal{L}}
\newcommand{\mcM}{\mathcal{M}}

\newcommand{\Aut}{\mathrm{Aut}}

\newcommand{\rot}{\mathrm{rot}}
\newcommand{\Leb}{\mathrm{Leb}}
\newcommand{\mrd}{\mathrm{d}}
\newcommand{\isorigid}{\mathrm{iso}\mbox{-}\mathrm{rigid}}
\newcommand{\equirigid}{\mathrm{equi}\mbox{-}\mathrm{rigid}}
\newcommand{\prigid}{\mathrm{p}\mbox{-}\mathrm{rigid}}
\newcommand{\id}{\mathrm{id}}

\newcommand{\bfA}{{\bf A}}
\newcommand{\bfB}{{\bf B}}

\newcommand{\TMS}{\Sigma_{\bfA}^+}
\newcommand{\tms}{\sigma_{\bfA}}

\newcommand{\bTMS}{\Sigma_{\bfB}^+}
\newcommand{\btms}{\sigma_{\bfB}}

\title[\null]{Continuity of isomorphisms applied to rigidity problems of entropy spectra}

\author[\null]{Katsukuni\ \ Nakagawa}

\address{K.\ Nakagawa\\
Graduate School of Advanced Science and Engineering\\
Hiroshima University\\
Higashi-Hiroshima 739-8526\\
Japan}
\email{ktnakagawa@hiroshima-u.ac.jp}

\begin{document}

\subjclass[2010]{37A35, 37B10, 37B40.}

\keywords{Isomorphism of measure-preserving systems, Continuity of isomorphisms, Gibbs measures, Entropy spectra, Rigidities of entropy spectra.}

\begin{abstract}
For a fixed topological Markov shift, we consider measure-preserving dynamical systems of Gibbs measures for 2-locally constant functions on the shift.
We also consider isomorphisms between two such systems.
We study the set of all 2-locally constant functions $f$ on the shift such that all those isomorphisms defined on the system associated with $f$ are induced from automorphisms of the shift.
We prove that this set contains a full-measure open set of the space of all 2-locally constant functions on the shift.
We apply this result to rigidity problems of entropy spectra and show that the strong non-rigidity occurs if and only if so does the weak non-rigidity.
\end{abstract}

\maketitle


\section{Introduction}
\label{sec:introduction}

Let $X_1$ and $X_2$ be topological spaces.
Let also $T_i:X_i\to X_i$ be a continuous map and $\mu_i$ a $T_i$-invariant Borel probability measure on $X_i$, for $i\in\{1,2\}$.
Recall the following definition:

\begin{definition}
\label{def:1377}
\normalfont
The two systems $(X_1,T_1,\mu_1)$ and $(X_2,T_2,\mu_2)$ are said to be \emph{isomorphic} if there exist two Borel sets $Y_1\subset X_1,\,Y_2\subset X_2$ and a bijection $\Phi:Y_1\to Y_2$ such that the following four conditions hold:
\begin{itemize}
\item For $i\in\{1,2\}$, $T_i(Y_i)\subset Y_i$ and $\mu_i(Y_i)=1$.
\item Both $\Phi:Y_1\to Y_2$ and $\Phi^{-1}:Y_2\to Y_1$ are Borel measurable. 
\item $\Phi\circ T_1=T_2\circ\Phi$ on $Y_1$.
\item $\mu_1=\mu_2\circ \Phi$.
\end{itemize}
\end{definition}
The map $\Phi:Y_1\to Y_2$ in Definition \ref{def:1377} is called an \emph{isomorphism} from $(X_1,T_1,\mu_1)$ to $(X_2,T_2,\mu_2)$.
For two isomorphisms $\Phi:Y_1\to Y_2$ and $\Psi:Z_1\to Z_2$ from $(X_1,T_1,\mu_1)$ to $(X_2,T_2,\mu_2)$, we call $\Psi$ is a \emph{restriction} of $\Phi$ if $Z_1\subset Y_1$ and $\Phi|_{Z_1}=\Psi$.
A homeomorphism $\tau:X_1\to X_2$ is called a \emph{topological conjugacy} from $(X_1,T_1)$ to $(X_2,T_2)$ if $\tau\circ T_1=T_2\circ\tau$ on $X_1$.
We give the following definition:

\begin{definition}
\normalfont
The two systems $(X_1,T_1,\mu_1)$ and $(X_2,T_2,\mu_2)$ are said to be \emph{equivalent} if there exists a topological conjugacy $\tau$ from $(X_1,T_1)$ to $(X_2,T_2)$ such that $\mu_1=\mu_2\circ\tau$.
\end{definition}

Clearly, two equivalent systems are isomorphic.
In this paper, we focus on the situation in which the converse holds, for a topological Markov shift. 

Let $N\ge2$ be an integer and $\bfA$ an $N\times N$ zero-one matrix.
We set
\[
\TMS=\{\omega=(\omega_n)_{n\in\N\cup\{0\}}\in\{1,\dots,N\}^{\N\cup\{0\}}:\bfA_{\omega_n\omega_{n+1}}=1,\ n\in\N\cup\{0\}\}
\]
and define the metric $d$ on $\TMS$ by 
\[
d(\omega,\omega')=\sum_{n=0}^{\infty}\frac{|\omega_n-\omega'_n|}{2^n}.
\]
Then, $\TMS$ is a compact metric space and the \emph{shift map} $\tms:\TMS\to\TMS$ defined by 
\[
(\tms\omega)_n=\omega_{n+1},\quad n\in\N\cup\{0\}
\]
is a continuous map. 
We call the topological dynamical system $(\TMS,\tms)$ a \emph{(one-sided) topological Markov shift} and $\bfA$ the \emph{transition matrix} of the shift.
We say that $\bfA$ is \emph{primitive} if there exists a positive integer $k$ such that all entries of the power $\bfA^k$ are positive.
A topological conjugacy from $(\TMS,\tms)$ to itself is called an \emph{automorphism} of $\TMS$.
For a H\"{o}lder continuous function $f:\TMS\to\R$, we denote by $\mu_f$ the (unique) Gibbs measure for $f$ (see Section \ref{sec:preliminaries} for the definition of a Gibbs measure). 
Let $\mcC$ be a set consisting of real-valued H\"{o}lder continuous functions on $\TMS$.
As in \cite{Nakagawa DS 2021}, for $f\in\mcC$, we define the subset $[f]_3$ of $\mcC$ by
\[
[f]_3=\{g\in\mcC:(\TMS,\tms,\mu_f)\ \mbox{and}\ (\TMS,\tms,\mu_g)\ \mbox{are isomorphic}\}.
\]
Moreover, we consider the following property (P) of $f$:
\[
(\mathrm{P})\ \ 
\begin{minipage}{0.85\textwidth}
For any $g\in[f]_3$ and isomorphism $\Phi$ from $(\TMS,\tms,\mu_f)$ to $(\TMS,\tms,\mu_g)$, there exist a restriction $\Psi$ of $\Phi$ and an automorphism $\bar{\Psi}$ of $\TMS$ such that $\bar{\Psi}$ is a unique continuous extension of $\Psi$ to $\TMS$.
\end{minipage}
\]
We set
\[
\mcC_0=\{f\in\mcC:f\ \mbox{satisfies the property (P)}\}.
\]
We denote by $L_2(\TMS,\R)$ the set of all real-valued 2-locally constant functions on $\TMS$ (see Section \ref{sec:preliminaries} for the definition of a 2-locally constant function).
Then, $L_2(\TMS,\R)$ is a finite-dimensional real vector space, and hence, becomes a real analytic manifold in a natural way.
We denote by $\Leb$ the Lebesgue measure on $L_2(\TMS,\R)$.
In this paper, we consider whether the set $\mcC_0$ is (topologically or measure-theoretically) large or not when $\mcC=L_2(\TMS,\R)$.
The following theorem is the main result of this paper and shows that the set $\mcC_0$ is both topologically and measure-theoretically large.

\begin{theorem}
\label{thm:7724}
Let $\bfA$ be an arbitrary primitive zero-one matrix of size $N(\ge2)$ and let $\mcC=L_2(\TMS,\R)$.
Then, $\mcC_0$ contains an open set $G$ of $\mcC$ such that $\Leb(\mcC\setminus G)=0$.
\end{theorem}

Notice that the set $G$ is dense in $\mcC$ if $\Leb(\mcC\setminus G)=0$. 

We apply Theorem \ref{thm:7724} to rigidity problems of entropy spectra.
We recall the definition of the entropy spectrum of a $\tms$-invariant Borel probability measure on $\TMS$. 

\begin{definition}
\normalfont
\label{def:524}
Let $\mu$ be a $\tms$-invariant Borel probability measure on $\TMS$. 
We set
\[
E_{\alpha}=\left\{\omega\in\TMS:\lim_{n\to\infty}-\frac{1}{n}\log\mu\left(\left\{\omega'\in\TMS:\omega_k=\omega'_k,\ k=0,\dots,n-1\right\}\right)=\alpha\right\}
\]
for $\alpha\in\R$ and define the function $\mcE^{(\mu)}:\R\to\R$ by
\[
\mcE^{(\mu)}(\alpha)=h(\tms|E_{\alpha}).
\]
Here, for $Z\subset\TMS$, we denote by $h(\tms|Z)$ the topological entropy of $Z$.
The function $\mcE^{(\mu)}$ is called the \emph{entropy spectrum} of $\mu$.
\end{definition}

In general, the entropy spectrum $\mcE^{(\mu)}$ has a wealth of information about the measure $\mu$, 
and a \emph{rigidity} is a phenomenon that the measure $\mu$ is `recovered' from its entropy spectrum $\mcE^{(\mu)}$.
However, in this paper, we consider `non-rigidity' phenomena.
See \cite{Barreira}, \cite{Nakagawa2016} and \cite{Nakagawa2017} for rigidity problems of dimension spectra on one-dimensional piecewise linear Markov maps.

Let $\mcC=L_2(\TMS,\R)$.
Again as in \cite{Nakagawa DS 2021}, for $f\in\mcC$, we define the two subsets $[f]_1$ and $[f]_2$ of $\mcC$ by
\begin{align*}
&[f]_1=\{g\in\mcC:\mcE^{(\mu_f)}=\mcE^{(\mu_g)}\},\\
&[f]_2=\{g\in\mcC:(\TMS,\tms,\mu_f)\ \mbox{and}\ (\TMS,\tms,\mu_g)\ \mbox{are equivalent}\}.
\end{align*}
Recall that two H\"{o}lder continuous functions $f,g:\TMS\to\R$ are said to be \emph{cohomologous with a constant} if $f-g=\varphi\circ\tms-\varphi+c$ holds for some H\"{o}lder continuous $\varphi:\TMS\to\R$ and $c\in\R$.
If two H\"{o}lder continuous functions $f,g:\TMS\to\R$ are cohomologous with a constant, then $\mu_f=\mu_g$, and hence, $[f]_i=[g]_i$ for $i\in\{1,2,3\}$.
Moreover, for $f\in\mcC$, it holds that
\begin{equation}
\label{eq:2927}
[f]_2\subset[f]_3\subset[f]_1.
\end{equation}
Indeed, $[f]_2\subset[f]_3$ is obvious and $[f]_3\subset[f]_1$ follows from \cite[Theorem 3.1]{Nakagawa DS 2021} which states that two isomorphic systems of Gibbs measures have the same entropy spectrum.
We set
\[
\mcC_{\isorigid}=\{f\in\mcC:[f]_1=[f]_3\},\quad\mcC_{\equirigid}=\{f\in\mcC:[f]_1=[f]_2\}
\]
and consider the following two non-rigidity phenomena:
\begin{alignat*}{2}
&\mbox{(SNR)}\ &&\mbox{The complement}\ \mcC\setminus\mcC_{\isorigid}\ \mbox{contains an open set}\ G\ \mbox{of}\ \mcC\ \mbox{such that}\\ 
&&&\Leb(\mcC\setminus G)=0.\\
&\mbox{(WNR)}\ &&\mbox{The complement}\ \mcC\setminus\mcC_{\equirigid}\ \mbox{contains an open set}\ G\ \mbox{of}\ \mcC\ \mbox{such that}\\
&&&\Leb(\mcC\setminus G)=0.
\end{alignat*}
We call the non-rigidity (SNR) (resp., (WNR)) the \emph{strong (resp., weak) non-rigidity}.
Notice that (\ref{eq:2927}) implies
\begin{equation}
\label{eq:1908}
\mcC_{\equirigid}\subset\mcC_{\isorigid}.
\end{equation}
Hence, if (SNR) occurs, then so does (WNR).
In fact, the converse also holds, and hence, we obtain the following theorem:

\begin{theorem}
\label{thm:5066}
Let $\bfA$ be an arbitrary primitive zero-one matrix of size $N(\ge2)$ and let $\mcC=L_2(\TMS,\R)$.
Then, the strong non-rigidity (SNR) occurs if and only if the weak non-rigidity (WNR) occurs.
\end{theorem}

The `if' part of Theorem \ref{thm:5066} is a direct consequence of Theorem \ref{thm:7724}.
For the sake of completeness, we present the proof of Theorem \ref{thm:5066}, in Section \ref{sec:non-rigidity} of this paper.
Moreover, in the same section, using (the `if' part of) Theorem \ref{thm:5066}, we show that the strong non-rigidity (SNR) occurs when
\begin{equation}
\label{eq:2978}
\bfA=\begin{pmatrix}
0&1&1&1\\
1&0&0&1\\
0&1&0&0\\
0&1&0&0
\end{pmatrix}.
\end{equation}
An analogue of Theorem \ref{thm:7724} was proved by \cite[Lemma 5.7]{Nakagawa DS 2021} for $\bfA$ such that at most one row of $\bfA$ has only one non-zero entry.
However, the matrix $\bfA$ in (\ref{eq:2978}) does not satisfy this condition since the third and fourth rows have only one non-zero entry.
Therefore, the strong non-rigidity  (SNR) for $\bfA$ in (\ref{eq:2978}) does not follow from the result of \cite{Nakagawa DS 2021}.
The non-rigidity phenomena strongly depend on the size $N$ of $\bfA$.
See \cite[Section 5]{Nakagawa DS 2021} for the explanation of the role of $N$ in rigidity problems.

This paper is organized as follows.
In Section \ref{sec:preliminaries}, we give preliminary definitions and basic facts. 
Especially, the construction of Gibbs measures via transfer operators is reviewed.
In Section \ref{sec:cond_continuity}, we give a sufficient condition for an isomorphism between two systems of Gibbs measures for 2-locally constant functions to be induced from a topological conjugacy (Theorem \ref{thm:5100}).
This condition is essential to the proof of Theorem \ref{thm:7724}, the main result of this paper.
In Section \ref{sec:main}, we prove Theorem \ref{thm:7724}.
In Section \ref{sec:non-rigidity}, we prove Theorem \ref{thm:5066}.
Moreover, we show that the strong non-rigidity (SNR) occurs for the transition matrix $\bfA$ in (\ref{eq:2978}) (Theorem \ref{thm:0426}).


\section{Preliminaries}
\label{sec:preliminaries}

Let $N\ge2$ be an integer and $\bfA$ an $N\times N$ primitive zero-one matrix.
An element of $\bigcup_{n\in\N\cup\{0\}}\{1,\dots,N\}^n$ is called a \emph{word}. 
For a word $w$ and $n\in\N\cup\{0\}$, we write $|w|=n$ if $w\in\{1,\dots,N\}^n$.
Thus, a word $w$ can be written in the form $w=w_0\cdots w_{|w|-1}$, where $w_0,\dots,w_{|w|-1}\in\{1,...,N\}$. 
The \emph{empty word} is the unique word $w$ with $|w|=0$. 
A word $w$ is said to be $\bfA$-\emph{admissible} if $|w|\le1$ or if $|w|\ge2$ and $\bfA_{w_kw_{k+1}}=1$ for $k\in\{0,\dots,|w|-2\}$. 
We denote by $W(\bfA)$ the set of all $\bfA$-admissible words.
Moreover, we denote by $E(\bfA)$ the set of all $\bfA$-admissible words $w$ with $|w|=2$.

We denote by $\mcM(\TMS)$ the set of all Borel probability measures on $\TMS$.
We give the definition of a Gibbs measure as follows:

\begin{definition}
\label{def:042}
\normalfont
Let $\mu\in\mcM(\TMS)$ be $\tms$-invariant and $f:\TMS\to\R$ continuous.
We call $\mu$ a \emph{Gibbs measure} for $f$ if there exist $C>0$ and $P\in\R$ such that the following inequality holds for $\omega\in\TMS$ and $n\in\N$:
\[
C^{-1}\le\frac{\mu\left(\left\{\omega'\in\TMS:\omega_k=\omega'_k,\ k=0,\dots,n-1\right\}\right)}{\exp(-nP+\sum_{k=0}^{n-1}f(\tms^k\omega))}\le C.
\]
\end{definition}
\begin{remark}
\label{rem:3322}
\normalfont
Let $\mu$ be a Gibbs measure.
By Definition \ref{def:042}, $\mu(G)>0$ for any non-empty open set $G$ of $\TMS$.
\end{remark}

For $\alpha>0$, we denote by $\mcH_{\alpha}(\TMS)$ the set of all complex-valued $\alpha$-H\"{o}lder continuous functions on $\TMS$.
We equip $\mcH_{\alpha}(\TMS)$ with the H\"{o}lder norm.
Then, $\mcH_{\alpha}(\TMS)$ is a complex Banach space.
For $f\in\mcH_{\alpha}(\TMS)$, we define the bounded linear operator $\mcL_f:\mcH_{\alpha}(\TMS)\to \mcH_{\alpha}(\TMS)$ by  
\[
(\mcL_f\phi)(\omega)=\sum_{\omega'\in\TMS:\,\tms\omega'=\omega}e^{f(\omega')}\phi(\omega')
\]
and call it the \emph{transfer operator} of $f$.

For a continuous function $f:\TMS\to\R$, we put
\[
\lambda_f=e^{P(\tms,f)},
\]
where $P(\tms,f)$ denotes the topological pressure of $f$ (for the definition, see, e.g., \cite[Chapter 3]{Parry-Pollicott}).
The following theorem is called the Ruelle-Perron-Frobenius theorem (for the proof, see, e.g., \cite[Theorem 2.2]{Parry-Pollicott}).

\begin{theorem}
\label{thm:410}
Let $\alpha>0$ and let $f\in\mcH_{\alpha}(\TMS)$ be real-valued.
\begin{enumerate}
\item[(i)]$\lambda_f$ is an eigenvalue of $\mcL_f:\mcH_{\alpha}(\TMS)\to \mcH_{\alpha}(\TMS)$ with algebraic multiplicity one.
\item[(ii)]There exist $h\in\mcH_{\alpha}(\TMS)$ and $\nu\in\mcM(\TMS)$ satisfying the following two properties:
\begin{itemize}
\item[(a)]$h>0$ and $\mcL_fh=\lambda_fh$.
\item[(b)]$\int\mcL_f\phi\,\mrd\nu=\lambda_f\int\phi\,\mrd\nu$ for all $\phi\in \mcH_{\alpha}(\TMS)$.
\end{itemize}
\end{enumerate}
\end{theorem}

Using the Ruelle-Perron-Frobenius theorem, we can construct Gibbs measures as follows:

\begin{theorem}
\label{thm:014}
Let $\alpha>0$ and let $f\in \mcH_{\alpha}(\TMS)$ be real-valued.
Let also $h$ and $\nu$ be as in Theorem \ref{thm:410}, with $\int h\,\mrd\nu=1$.
Then, the Borel probability measure $\mu$ on $\TMS$ defined by $\mrd\mu=h\,\mrd\nu$ is a unique Gibbs measure for $f$.
\end{theorem}

For the proof of Theorem \ref{thm:014}, see, e.g., \cite[Corollary 3.2.1]{Parry-Pollicott}.
For a H\"{o}lder continuous function $f:\TMS\to\R$, we denote by $\mu_f$ the unique Gibbs measure for $f$.

For $\nu\in\mcM(\TMS)$, we say that $\nu$ is \emph{non-singular} if $\nu\circ\tms^{-1}$ is absolutely continuous with respect to $\nu$.
Let $\nu\in\mcM(\TMS)$ be non-singular.
We define the finite Borel measure $\nu^{(1)}$ on $\TMS$ by
\[
\nu^{(1)}(B)=\sum_{i=1}^N\nu(\tms(B\cap\{\omega\in\TMS:\omega_0=i\})).
\]
It is easy to see that $\nu$ is absolutely continuous with respect to $\nu^{(1)}$.
The Radon-Nikodym derivative
\[
J_{\nu}=\frac{\mrd\nu}{\mrd\nu^{(1)}}
\]
is called the \emph{Jacobian} of $\nu$.
The proof of the following proposition is easy, so we omit it.

\begin{proposition}
\label{prop:1087}
Let $f:\TMS\to\R$ be H\"{o}lder continuous and $\nu$ as in Theorem \ref{thm:410}.
Then, the following two assertions hold:
\begin{enumerate}
\item[(i)] $\nu$ is non-singular.
\item[(ii)] $\nu$ and $\nu^{(1)}$ are absolutely continuous with respect to each other.
Moreover, $J_{\nu}=\lambda_f^{-1}\exp(f),\ \nu\mbox{-a.e.}$
\end{enumerate}
\end{proposition}

For a H\"{o}lder continuous function $f:\TMS\to\R$, we write
\[
\widehat{f}=f+\log\frac{\,h}{\lambda_f(h\circ\tms)},
\]
where $h$ is as in Theorem \ref{thm:410}.
It is easy to see that the definition of $\widehat{f}$ is independent of the choice of $h$ and that $\lambda_{\widehat{f}}=1$.
Moreover, $\widehat{f}$ and $f$ are cohomologous with a constant, and hence, $\mu_{\widehat{f}}=\mu_f$.
The next lemma plays an important role in the proof of Theorem \ref{thm:5100} below.

\begin{lemma}
\label{lem:0477}
Let $f:\TMS\to\R$ be H\"{o}lder continuous.
Moreover, let $\bfB$ be another primitive zero-one matrix and $g:\bTMS\to\R$ H\"{o}lder continuous.
If $(\TMS,\tms,\mu_f)$ and $(\bTMS,\btms,\mu_g)$ are isomorphic, then the following two assertions hold:
\begin{enumerate}
\item[(i)]For any isomorphism $\Phi$ from $(\TMS,\tms,\mu_f)$ to $(\bTMS,\btms,\mu_g)$, we have $\widehat{f}=\widehat{g}\circ\Phi,\ \mu_f\mbox{-}a.e.$
\item[(ii)]$\widehat{f}(\TMS)=\widehat{g}(\bTMS)$.
\end{enumerate}
\end{lemma}
\begin{proof}
(i)\ 
Take $\alpha>0$ so that $f\in\mcH_{\alpha}(\TMS)$ and $g\in\mcH_{\alpha}(\bTMS)$.
It is easy to see that $\mcL_{\widehat{f}}1=1$ and $\int\mcL_{\widehat{f}}\phi\,\mrd\mu_f=\int\phi\,\mrd\mu_f$ for $\phi\in \mcH_{\alpha}(\TMS)$.
The analogues for $g$ also hold.
Hence, we see from Proposition \ref{prop:1087} (ii) that $\exp(\widehat{f})=J_{\mu_f}=J_{\mu_g}\circ\Phi=\exp(\widehat{g}\circ\Phi),\ \mu_f\mbox{-}a.e.$
Thus, (i) follows.

(ii)\ From (i), there exists a Borel set $A$ of $\TMS$ with $\mu_f(A)=1$ such that $\widehat{f}=\widehat{g}\circ\Phi$ on $A$.
We have $\widehat{f}(A)=\widehat{g}(\Phi(A))$.
Moreover, from $\mu_f(A)=\mu_g(\Phi(A))=1$ and Remark \ref{rem:3322}, $A$ and $\Phi(A)$ are dense in $\TMS$ and $\bTMS$, respectively.
Hence, by the continuity of $\widehat{f}$ and $\widehat{g}$, we have $\widehat{f}(\TMS)=\widehat{g}(\bTMS)$.
Thus, (ii) follows.
\end{proof}

We define an important class of H\"{o}lder continuous functions on $\TMS$ called locally-constant functions.

\begin{definition}
\normalfont
Let $n\in\N\cup\{0\}$.
A function $\phi:\TMS\to\C$ is said to be $n$-\emph{locally constant} if $\phi(\omega)=\phi(\omega')$ for any $\omega,\omega'\in\TMS$ with $\omega_0\cdots\omega_{n-1}=\omega'_0\cdots\omega'_{n-1}$.
\end{definition}

We denote by $L_n(\TMS,\R)$ the set of all real-valued $n$-locally constant functions on $\TMS$.
It is obvious that $L_n(\TMS,\R)\subset\mcH_{\alpha}(\TMS)$ for any $n\in\N\cup\{0\}$ and $\alpha>0$.
In this paper, we consider only 2-locally constant functions. 

For 2-locally constant functions, we can give a more specific description of their transfer operators, using matrix theory.
We denote by $M_{\bfA}$ the set of all $N\times N$ non-negative matrices $A$ such that, for $i,j\in\{1,\dots,N\}$, $A_{ij}>0$ if and only if $\bfA_{ij}=1$.
A non-negative square matrix $Q$ is called a \emph{column stochastic matrix} if the entries in each column of $Q$ sum to 1.
We also denote by $\mcM_{\bfA}$ the set of all column stochastic matrices in $M_{\bfA}$.
Moreover, for $a=(a_1,\dots,a_N)\in\R^N$, we denote by $a^{\top}$ the transpose of $a$.

\begin{theorem}
\label{thm:790}
For $A\in M_{\bfA}$, the following two assertions hold:
\begin{enumerate}
\item[(i)]There exists a unique positive eigenvalue $\lambda$ of $A$ such that $\lambda>|\eta|$ for any other eigenvalue $\eta\in\C$ of $A$.
Moreover, $\lambda$ has algebraic multiplicity one.
\item[(ii)]Let $\lambda$ be as in (i).
Then, there exist $p,v\in\R^N$ satisfying the following two properties:
\begin{itemize}
\item[(a)] $p_i>0$ and $v_i>0$ for $i\in\{1,\dots,N\}$.
\item[(b)] $p$ and $v^{\top}$ are left and right eigenvectors of $\lambda$, respectively.
\end{itemize}
\end{enumerate}
\end{theorem}
For the proof of Theorem \ref{thm:790}, see, e.g., \cite[Theorem 1.1]{Seneta}.
We call the above unique eigenvalue $\lambda$ the \emph{Perron root} of $A$.

\begin{remark}
\label{rem:9599}
\normalfont
Let $Q\in\mcM_{\bfA}$.
Then, the Perron root of $Q$ is 1 and $(1,\dots,1)Q=(1,\dots,1)$.
\end{remark}

For $f\in L_2(\TMS,\R)$, we define the matrix $A(f)\in M_{\bfA}$ by
\[
A(f)_{ij}=\bfA_{ij}\exp(f|_{\{\omega\in\TMS:\,\omega_0=i,\ \omega_1=j\}}),\quad i,j\in\{1,\dots,N\}.
\]

\begin{proposition}
\label{prop:2018}
Let $f\in L_2(\TMS,\R)$ and let $\alpha>0$.
\begin{enumerate}
\item[(i)] Let $\lambda$ be an eigenvalue of $A(f)$ and $p\in\C^N$ a left eigenvector of $\lambda$.
We define the 1-locally constant function $h:\TMS\to\C$ by $h(\omega)=p_{\omega_0}$.
Then, $\mcL_fh=\lambda h$, and hence, $\lambda$ is an eigenvalue of $\mcL_f:\mcH_{\alpha}(\TMS)\to\mcH_{\alpha}(\TMS)$.
\item[(ii)] The Perron root of $A(f)$ is equal to the eigenvalue $\lambda_f$ of $\mcL_f:\mcH_{\alpha}(\TMS)\to\mcH_{\alpha}(\TMS)$.
\end{enumerate}
\end{proposition}

Let $f\in L_2(\TMS,\R)$.
We denote by $p(f)$ the unique left eigenvector of $\lambda_f$ with $\sum_{i=1}^Np(f)_i=1$.
Then, from Proposition \ref{prop:2018},
\[
\widehat{f}(\omega)=f(\omega)+\log\frac{p(f)_{\omega_0}}{\lambda_f\,p(f)_{\omega_1}}.
\]
Moreover, we define $Q(f)\in M_{\bfA}$ by 
\begin{equation}
\label{eq:5211}
Q(f)_{ij}=A(\widehat{f})_{ij}=\lambda_f^{-1}p(f)_j^{-1}p(f)_iA(f)_{ij},\quad i,j\in\{1,\dots,N\}.
\end{equation}
It is easy to see that $Q(f)\in\mcM_{\bfA}$, i.e., $Q(f)$ is a column stochastic matrix.

We define the function $\delta_{\bfA}:\{1,\dots,N\}\to\N\cup\{0\}$ by
\[
\delta_{\bfA}(j)=\#\{i\in\{1,\dots,N\}:ij\in E(\bfA)\}.
\]
In other words, $\delta_{\bfA}(j)$ denotes the number of non-zero entries in the $j$th column of $\bfA$, for $j\in\{1,\dots,N\}$.
Since $\bfA$ is primitive, $\delta_{\bfA}(j)\ge1$ for any $j\in\{1,\dots,N\}$.
We set
\[
V_0(\bfA)=\{j\in\{1,\dots,N\}:\delta_{\bfA}(j)\ge2\},\ \ E_0(\bfA)=\{ij\in E(\bfA):j\in V_0(\bfA)\}.
\]
We define the set $\mcG_{\bfA}$ of column stochastic matrices by
\begin{equation}
\label{eq:3182}
\mcG_{\bfA}=\{Q\in\mcM_{\bfA}:Q_{ij}\ne Q_{kl}\ \mbox{for any}\ ij,kl\in E_0(\bfA)\ \mbox{with}\ ij\ne kl\}.
\end{equation}

\begin{remark}
\label{rem:1085}
\normalfont
Let $Q\in\mcM_{\bfA}$.
It is easy to see that the following two assertions hold:
\begin{enumerate}
\item[(i)]$E_0(\bfA)=\{ij\in E(\bfA):Q_{ij}\in(0,1)\}$ holds.
\item[(ii)]$\#\{Q_{ij}:ij\in E_0(\bfA)\}\le\#E_0(\bfA)$ holds.
Moreover, the equality $\#E_0(\bfA)=\#\{Q_{ij}:ij\in E_0(\bfA)\}$ holds if and only if $Q\in\mcG_{\bfA}$.
\end{enumerate}
\end{remark}

The following lemma plays a key role in the proof of Theorem \ref{thm:5100} below.

\begin{lemma}
\label{lem:1723}
Let $w,w'\in W(\bfA)$ satisfy $|w|\ge2$ and $|w|=|w'|$.
Let also $Q\in\mcG_{\bfA}$.
If $w_{|w|-1}\in V_0(\bfA)$ and $Q_{w_kw_{k+1}}=Q_{w'_kw'_{k+1}}$ for all $k\in\{0,\dots,|w|-2\}$, then $w=w'$.
\end{lemma}
\begin{proof}
First, we show that $w_{|w|-2}w_{|w|-1}=w'_{|w|-2}w'_{|w|-1}$.
Since $w_{|w|-1}\in V_0(\bfA)$, we have $w_{|w|-2}w_{|w|-1}\in E_0(\bfA)$.
Thus, by Remark \ref{rem:1085} (i), we have $Q_{w'_{|w|-2}w'_{|w|-1}}=Q_{w_{|w|-2}w_{|w|-1}}\in(0,1)$.
Again by Remark \ref{rem:1085} (i), we have $w'_{|w|-2}w'_{|w|-1}\in E_0(\bfA)$.
Since $Q\in\mcG_{\bfA}$, we have $w'_{|w|-2}w'_{|w|-1}=w_{|w|-2}w_{|w|-1}$.

Next, we show that, for $k\in\{1,\dots,|w|-1\}$, $w_{k-1}=w'_{k-1}$ if $w_k=w'_k$.
We consider the case in which $w_k\in V_0(\bfA)$.
Then, $w_{k-1}w_k\in E_0(\bfA)$ and $w'_{k-1}w'_k=w'_{k-1}w_k\in E_0(\bfA)$.
Thus, by $Q_{w'_{k-1}w'_k}=Q_{w_{k-1}w_k}$ and $Q\in\mcG_{\bfA}$, we have $w'_{k-1}w'_k=w_{k-1}w_k$.
We consider the case in which $w_k\notin V_0(\bfA)$.
Then, $\delta_{\bfA}(w_k)=1$.
From this and $\bfA_{w'_{k-1}w_k}=\bfA_{w'_{k-1}w'_k}=\bfA_{w_{k-1}w_k}=1$, we have $w'_{k-1}=w_{k-1}$.
\end{proof}

A word $w$ with $|w|\ge2$ is called a \emph{cycle} if $w_0=w_{|w|-1}$, and a cycle $w$ is said to be \emph{simple} if $w_k\ne w_l$ for $k,l\in\{0,\dots,|w|-2\}$ with $k\ne l$.

\begin{lemma}
\label{lem:305a}
Let $w\in W(\bfA)$ be a cycle.
Then, there exists $k\in\{0,\dots,|w|-1\}$ such that $w_k\in V_0(\bfA)$.
\end{lemma}
\begin{proof}
Without loss of generality, we may assume that $w$ is simple and that $\{w_0,\dots,w_{|w|-2}\}=\{1,\dots,|w|-1\}$.
Let $\{1,\dots,|w|-1\}\cap V_0(\bfA)=\emptyset$. 
Then, it follows from the definition of $V_0(\bfA)$ that $\bfA$ is of the form $\bfA=\left(
\begin{array}{c|c}
\begin{matrix}
A\\
O
\end{matrix}&*
\end{array}
\right)$, where $A$ is a $(|w|-1)\times(|w|-1)$ permutation matrix and $O$ is the $(N-|w|+1)\times(|w|-1)$ zero matrix (a square zero-one matrix is called a \emph{permutation matrix} if each row and column has only one non-zero entry).
Hence, for any $k\in\N$, the power $\bfA^k$ has at least one entry equal to 0.
Since $\bfA$ is primitive, we obtain a contradiction.
\end{proof}


\section{Sufficient condition for continuity of isomorphism}
\label{sec:cond_continuity}

Let $N\ge2$ be an integer and $\bfA$ an $N\times N$ primitive zero-one matrix.
Let also $\bfB$ be another primitive zero-one matrix.
The aim of this section is to prove the following theorem which gives a sufficient condition for an isomorphism between two systems of Gibbs measures for 2-locally constant functions to be induced from a topological conjugacy.
\begin{theorem}
\label{thm:5100}
Let $\#E_0(\bfA)=\#E_0(\bfB)$ and let $f\in L_2(\TMS,\R)$ satisfy $Q(f)\in\mcG_{\bfA}$.
Let also $g\in L_2(\bTMS,\R)$.
Assume that $(\TMS,\tms,\mu_f)$ and $(\bTMS,\btms,\mu_g)$ are isomorphic.
Then,  any isomorphism $\Phi$ from $(\TMS,\tms,\mu_f)$ to $(\bTMS,\btms,\mu_g)$, there exist a restriction $\Psi$ of $\Phi$ and a topological conjugacy $\bar{\Psi}$ from $(\TMS,\tms)$ to $(\bTMS,\btms)$ such that $\bar{\Psi}$ is a unique continuous extension of $\Psi$ to $\TMS$.
\end{theorem}

Let $f\in L_2(\TMS,\R)$ and $g\in L_2(\bTMS,\R)$.
In the rest of this section, we always assume that $(\TMS,\tms,\mu_f)$ and $(\bTMS,\btms,\mu_g)$ are isomorphic.
Let $\Phi:Y_1\to Y_2$ be an isomorphism from $(\TMS,\tms,\mu_f)$ to $(\bTMS,\btms,\mu_g)$.
We begin the proof of Theorem \ref{thm:5100} with the following lemma:

\begin{lemma}
\label{lem:3725}
The following two assertions hold:
\begin{enumerate}
\item[(i)] There exists a Borel set $Z_1$ of $\TMS$ such that the following two conditions hold:
\begin{enumerate}
\item[(a)] $Z_1\subset Y_1,\ \tms(Z_1)\subset Z_1$ and $\mu_f(Z_1)=1$.
Moreover, $\Phi\circ\tms=\btms\circ\Phi$ on $Z_1$.
\item[(b)]$Q(f)_{\omega_n\omega_{n+1}}=Q(g)_{(\Phi\omega)_n(\Phi\omega)_{n+1}}$ for $\omega\in Z_1$ and $n\in\N\cup\{0\}$.
\end{enumerate} 
\item[(ii)] We have $\{Q(f)_{ij}:ij\in E_0(\bfA)\}=\{Q(g)_{ij}:ij\in E_0(\bfB)\}$.
\end{enumerate}
\end{lemma}
\begin{proof}
(i) and (ii) follow from (i) and (ii) of Lemma \ref{lem:0477}, respectively.
\end{proof}

The next lemma is a key step in the proof of Theorem \ref{thm:5100}.

\begin{lemma}
\label{lem:4320}
Let $Z_1$ be as in Lemma \ref{lem:3725} (i).
If $Q(g)\in\mcG_{\bfB}$, then, for any $n\in\N\cup\{0\}$ and $\omega,\omega'\in Z_1$ with $\omega_0\cdots\omega_{N+n}=\omega'_0\cdots\omega'_{N+n}$, we have $(\Phi\omega)_0\cdots(\Phi\omega)_n=(\Phi\omega')_0\cdots(\Phi\omega')_n$.
\end{lemma}
\begin{proof}
Without loss of generality, we may assume that $n=0$.
From Lemma \ref{lem:305a}, there exists $k_0\in\{1,\dots,N\}$ such that $\omega_{k_0}\in V_0(\bfA)$.
Moreover, 
\begin{equation}
\label{eq:9175}
\begin{aligned}
Q(g)_{(\Phi\omega)_k(\Phi\omega)_{k+1}}=&Q(f)_{\omega_k\omega_{k+1}}\\
=&Q(f)_{\omega'_k\omega'_{k+1}}=Q(g)_{(\Phi\omega')_k(\Phi\omega')_{k+1}},\quad k\in\{0,\dots,k_0-1\}.
\end{aligned}
\end{equation}
From (\ref{eq:9175}) and Remark \ref{rem:1085} (i), $Q(g)_{(\Phi\omega)_{k_0-1}(\Phi\omega)_{k_0}}=Q(f)_{\omega_{k_0-1}\omega_{k_0}}\in(0,1)$.
Hence, again from Remark \ref{rem:1085} (i), $(\Phi\omega)_{k_0}\in V_0(\bfB)$.
Thus, from (\ref{eq:9175}) and Lemma \ref{lem:1723}, $(\Phi\omega)_0\cdots(\Phi\omega)_{k_0}=(\Phi\omega')_0\cdots(\Phi\omega')_{k_0}$.
In particular, $(\Phi\omega)_0=(\Phi\omega')_0$.
\end{proof}

Let $X_1$ and $X_2$ be topological spaces.
Let also $T_i:X_i\to X_i$ be a continuous map, for $i\in\{1,2\}$.
A continuous surjection $\tau:X_1\to X_2$ is called a \emph{topological semi-conjugacy} from $(X_1,T_1)$ to $(X_2,T_2)$ if $\tau\circ T_1=T_2\circ\tau$ on $X_1$.
The following lemma is a direct consequence of Lemma \ref{lem:4320}.

\begin{lemma}
\label{lem:8071}
If $Q(g)\in\mcG_{\bfB}$, then the following two assertions hold:
\begin{enumerate}
\item[(i)] There exist a restriction $\Psi$ of $\Phi$ and a topological semi-conjugacy $\bar{\Psi}$ from $(\TMS,\tms)$ to $(\bTMS,\btms)$ such that $\bar{\Psi}$ is a unique continuous extension of $\Psi$ to $\TMS$.
\item[(ii)]Let $\bar{\Psi}$ be as in (i).
If $Q(f)\in\mcG_{\bfA}$, then $\bar{\Psi}$ is a topological conjugacy from $(\TMS,\tms)$ to $(\bTMS,\btms)$. 
\end{enumerate}
\end{lemma}
\begin{proof}
(i)\ 
Let $Z_1$ be as in Lemma \ref{lem:3725} (i) and let $\Psi=\Phi|_{Z_1}$.
Then, $\Psi$ is a restriction of $\Phi$.
Lemma \ref{lem:4320} implies that $\Psi$ is uniformly continuous on $Z_1$.
Moreover, Remark \ref{rem:3322} implies that $Z_1$ is dense in $\TMS$.
Thus, there exists a unique continuous extension $\bar{\Psi}$ of $\Psi$ to $\TMS$.
It is easy to see that $\bar{\Psi}$ is a topological semi-conjugacy from $(\TMS,\tms)$ to $(\bTMS,\btms)$.

(ii)\ 
Since $Q(f)\in\mcG_{\bfA}$, from (i), there exist a restriction $\Upsilon$ of $\Phi^{-1}$ and a topological semi-conjugacy $\bar{\Upsilon}$ from $(\bTMS,\btms)$ to $(\TMS,\tms)$ such that $\bar{\Upsilon}$ is a unique continuous extension of $\Upsilon$ to $\bTMS$.
It is easy to see that there exists a Borel set $W$ of $\TMS$ with $\mu_f(W)=1$ such that $\bar{\Upsilon}\bar{\Psi}=\id$ on $W$.
From Remark \ref{rem:3322}, $W$ is dense in $\TMS$.
Thus, it follows from the continuity of $\bar{\Psi}$ and $\bar{\Upsilon}$ that $\bar{\Upsilon}\bar{\Psi}=\id$ on $\TMS$.
Similarly, $\bar{\Psi}\bar{\Upsilon}=\id$ on $\bTMS$.
Thus, $\bar{\Psi}$ is a homeomorphism from $\TMS$ to $\bTMS$, and hence, is a topological conjugacy from $(\TMS,\tms)$ to $(\bTMS,\btms)$. 
\end{proof}

For the proof of Theorem \ref{thm:5100}, we need the next lemma.

\begin{lemma}
\label{lem:6653}
Let $\#E_0(\bfA)=\#E_0(\bfB)$.
Then, $Q(f)\in\mcG_{\bfA}$ if and only if $Q(g)\in\mcG_{\bfB}$.
\end{lemma}
\begin{proof}
It is enough to prove the `only if' part. 
Since $Q(f)\in\mcG_{\bfA}$, we have $\#\{Q(g)_{ij}:ij\in E_0(\bfB)\}=\#E_0(\bfA)$ by Lemma \ref{lem:3725} (ii) and Remark \ref{rem:1085} (ii).
Thus, we have $Q(g)\in\mcG_{\bfB}$ by $\#E_0(\bfA)=\#E_0(\bfB)$ and Remark \ref{rem:1085} (ii).
\end{proof}

Now, we are ready to prove Theorem \ref{thm:5100}.

\begin{proof}[Proof of Theorem \ref{thm:5100}]
It follows from Lemma \ref{lem:6653} that $Q(g)\in\mcG_{\bfB}$.
Thus, the assertion follows from Lemma \ref{lem:8071} (ii).
\end{proof}


\section{Proof of Theorem \ref{thm:7724}}
\label{sec:main}

In this section, we prove Theorem \ref{thm:7724} which is the main result of this paper.

Let $N\ge2$ be an integer and $\bfA$ an $N\times N$ primitive zero-one matrix.
Let $\mcC=L_2(\TMS,\R)$.
Recall from Section \ref{sec:introduction} that $\Leb$ denotes the Lebesgue measure on the real analytic manifold $\mcC$.
We set
\begin{equation}
\label{eq:3048}
G=\{f\in\mcC:Q(f)\in\mcG_{\bfA}\}.
\end{equation}

We begin the proof of Theorem \ref{thm:7724} with the following lemma:

\begin{lemma}
\label{lem:0036}
Let $G$ be as in (\ref{eq:3048}).
Then, $G$ is open in $\mcC$ and $\Leb(\mcC\setminus G)=0$.
\end{lemma}
\begin{proof}
The proof is the same as that of \cite[Lemma 5.6]{Nakagawa DS 2021}.
For $ij,kl\in E_0(\bfA)$ with $ij\neq kl$, we set
\[
G(ij;kl)=\left\{f\in\mcC:\frac{A(f)_{ij}}{A(f)_{kl}}-\frac{p(f)_j\,p(f)_k}{p(f)_i\,p(f)_l}\neq0\right\}.
\]
Then, by (\ref{eq:5211}), $G=\bigcap_{ij,\,kl\in E_0(\bfA):\,ij\neq kl}G(ij;kl)$.
Thus, it is enough to prove that $G(ij;kl)$ is open in $\mcC$ and that $\Leb(\mcC\setminus G(ij;kl))=0$, for $ij,kl\in E_0(\bfA)$ with $ij\neq kl$.

Fix $ij,kl\in E_0(\bfA)$ with $ij\ne kl$.
We define the map $F:\mcC\to\R$ by $F(f)=(A(f)_{ij}/A(f)_{kl})-[(p(f)_j\,p(f)_k)/(p(f)_i\,p(f)_l)]$.
Then, from \cite[Lemma 5.2]{Nakagawa DS 2021}, $F$ is real analytic, and hence, is continuous.
Thus, $G(ij;kl)$ is open in $\mcC$.
Moreover, it follows from the definition of $E_0(\bfA)$ that there exist $f,g\in\mcC$ such that $A(f)=Q(f),\ A(g)=Q(g)$ and $A(f)_{ij}/A(f)_{kl}\ne A(g)_{ij}/A(g)_{kl}$.
Since both $A(f)$ and $A(g)$ are column stochastic matrices, from Remark \ref{rem:9599}, $p(f)=p(g)=(1/N,\dots,1/N)$.
Thus, $F(f)\ne F(g)$, and hence, $F$ is a non-zero real analytic map.
The zero set of a non-zero real analytic map is contained in a countable union of hypersurfaces; see, e.g., \cite[Claims 1 and 2]{Mityagin}.
Therefore, $\mcC\setminus G(ij;kl)=\{f\in\mcC:F(f)=0\}$ is contained in a countable union of hypersurfaces in $\mcC$, and hence, $\Leb(\mcC\setminus G(ij;kl))=0$.
We obtain the desired result.
\end{proof}

The following lemma is a direct consequence of Theorem \ref{thm:5100}.

\begin{lemma}
\label{lem:4054}
Let $G$ be as in (\ref{eq:3048}).
Then, any $f\in G$ satisfies the property (P).
\end{lemma}
\begin{proof}
Let $g\in[f]_3$ and let $\Phi$ be an isomorphism from $(\TMS,\tms,\mu_f)$ to $(\TMS,\tms,\mu_g)$.
Then, $Q(f)\in\mcG_{\bfA}$, and hence, from Theorem \ref{thm:5100}, there exist a restriction $\Psi$ of $\Phi$ and an automorphism $\bar{\Psi}$ of $\TMS$ such that $\bar{\Psi}$ is a unique continuous extension of $\Psi$ to $\TMS$.
Hence, $f$ satisfies the property (P).
\end{proof}

Now, we are ready to prove Theorem \ref{thm:7724}.

\begin{proof}[Proof of Theorem \ref{thm:7724}]
From Lemmas \ref{lem:0036} and \ref{lem:4054}, the assertion holds for $G$ in (\ref{eq:3048}).
\end{proof}
%


\section{Example of strong non-rigidity}
\label{sec:non-rigidity}

Let $N\ge2$ be an integer and $\bfA$ an $N\times N$ primitive zero-one matrix.
Let also $\mcC=L_2(\TMS,\R)$.
The first aim of this section is to prove Theorem \ref{thm:5066} in Section \ref{sec:introduction}.

\begin{proof}[Proof of Theorem \ref{thm:5066}]
As mentioned in Section \ref{sec:introduction}, the `only if' part follows from (\ref{eq:1908}).
We prove the `if' part.
Assume that (WNR) occurs.
Then, there exists an open set $G'$ of $\mcC$ such that $G'\subset\mcC\setminus\mcC_{\equirigid}$ and $\Leb(\mcC\setminus G')=0$.
Let $G$ be as in Theorem \ref{thm:7724} and let $G''=G\cap G'$.
Then, $G''$ is open in $\mcC$ and $\Leb(\mcC\setminus G'')=0$.
Moreover, if $f\in G''$, then $f$ satisfies the property (P), and hence, $[f]_2=[f]_3$.
Therefore, $G''\subset\mcC\setminus\mcC_{\isorigid}$, and hence, (SNR) occurs.
\end{proof}

The second aim of this section is to prove the following theorem which states that the strong non-rigidity (SNR) occurs for the matrix $\bfA$ in (\ref{eq:2978}).

\begin{theorem}
\label{thm:0426}
Let $\bfA$ be as in (\ref{eq:2978}).
Then, the strong non-rigidity (SNR) occurs.
\end{theorem}

We denote by $\Aut(\TMS)$ the set of all automorphisms of $\TMS$.
We begin the proof of Theorem \ref{thm:0426} with the following lemma:

\begin{lemma}
\label{lem:6668}
Let $\bfA$ be as in (\ref{eq:2978}).
Then, $\Aut(\TMS)=\{\id\}$.
\end{lemma}
\begin{proof}
Let $H$ be the (unlabeled) directed graph whose adjacency matrix is $\bfA$.
Then, it is easy to see that the automorphism group of $H$ is trivial.
Moreover, since $\bfA$ has no repeated columns, the \emph{total amalgamation} (for the definition, see \cite[Chapters 2 and 3]{Kitchens}) of $\bfA$ is $\bfA$ itself.
Thus, by \cite[Corollary 3.2.8]{Kitchens}, the lemma follows. 
\end{proof}

The next lemma implies that the weak non-rigidity (WNR) occurs for $\bfA$ in (\ref{eq:2978}).

\begin{lemma}
\label{lem:6168}
Let $\bfA$ be as in (\ref{eq:2978}) and let $f\in\mcC$.
If $Q(f)_{32}\ne Q(f)_{14}Q(f)_{42}$, then there exists $g\in\mcC$ satisfying the following two conditions:
\begin{enumerate}
\item[(i)]$\mcE^{(\mu_f)}=\mcE^{(\mu_g)}$.
\item[(ii)]$(\TMS,\tms,\mu_f)$ and $(\TMS,\tms,\mu_g)$ are not equivalent.
\end{enumerate}
\end{lemma}
\begin{proof}
Let
\[
Q(f)=\begin{pmatrix}
0&a_1&1&b_1\\
1&0&0&b_2\\
0&a_2&0&0\\
0&a_3&0&0
\end{pmatrix}
\]
and take $g\in\mcC$ so that
\[
Q(g)=\begin{pmatrix}
0&a_1&1&a_2/c\\
1&0&0&a_3b_2/c\\
0&a_3b_1&0&0\\
0&c&0&0
\end{pmatrix},
\]
where $c=1-a_1-a_3b_1$.

We prove (i).
For $Q\in\mcM_{\bfA}$ and $q\in\R$, we define $Q_q\in\ M_{\bfA}$ by $(Q_q)_{ij}=\bfA_{ij}Q_{ij}^q,\ i,j\in\{1,\dots,N\}$. 
It is easy to see that, for $q\in\R$, the two characteristic polynomials of $Q(f)_q$ and $Q(g)_q$ are the same.
Therefore, $\mcE^{(\mu_f)}=\mcE^{(\mu_g)}$ (see, e.g., \cite[Theorem A4.2]{Pesin}).

We prove (ii).
Let $\omega$ be the periodic point $132132\cdots$ with period 3.
Then, $\sum_{k=0}^2\widehat{f}(\tms^k\omega)=a_2\ne a_3b_1=\sum_{k=0}^2\widehat{g}(\tms^k\omega)$, and hence, $\widehat{f}$ and $\widehat{g}$ are not cohomologous with a constant.
Therefore, $\mu_f=\mu_{\widehat{f}}\ne\mu_{\widehat{g}}=\mu_g$, and hence, by Lemma \ref{lem:6668}, the two systems $(\TMS,\tms,\mu_f)$ and $(\TMS,\tms,\mu_g)$ are not equivalent.
\end{proof}

We are ready to prove Theorem \ref{thm:0426}.

\begin{proof}[Proof of Theorem \ref{thm:0426}]
By the same argument as that in the proof Lemma \ref{lem:0036}, we can prove that the set $G=\{f\in\mcC:Q(f)_{32}\ne Q(f)_{14}Q(f)_{42}\}$ is open in $\mcC$ and that $\Leb(\mcC\setminus G)=0$.
Thus, it follows from Lemma \ref{lem:6168} that (WNR) occurs.
Therefore, by Theorem \ref{thm:5066}, we conclude that (SNR) occurs.
\end{proof}

\begin{remark}
\normalfont
Let $\bfA$ be as in (\ref{eq:2978}).
We can give a rigidity to the shift $(\TMS,\tms)$ as follows.

Recall the definition of a cycle given in Section \ref{sec:preliminaries}.
For a cycle $w$, we define the new cycle $\rot(w)$ by $\rot(w)=w_1\cdots w_{|w|-1}w_1$.
Two cycles $w,w'$ are identified if $w'=\rot^k(w)$ for some $k\in\N\cup\{0\}$.
Moreover, for two cycles $w,w'$, we write $w\cap w'\ne\emptyset$ if $w_k=w'_l$ for some $k\in\{0,\dots,|w|-1\}$ and $l\in\{0,\dots,|w'|-1\}$.

We denote by $S_{\bfA}$ the set of all $\bfA$-admissible simple cycles.
Then, $S_{\bfA}=\{121,242,1321,1421\}$, and hence, the following assertion holds:
\begin{equation}
\label{eq:6637}
\mbox{for any}\ w,w'\in S_{\bfA},\ w\cap w'\ne\emptyset\ \mbox{and}\ \frac{|w|-1}{|w'|-1}<2.
\end{equation}
Let $Q\in\mcM_{\bfA}$.
For $q\in\R$, we define the matrix $Q_q\in M_{\bfA}$ by $(Q_q)_{ij}=\bfA_{ij}Q_{ij}^q,\ i,j\in\{1,\dots,N\}$.
Moreover, we define the one-parameter family $\Phi_Q$ of characteristic polynomials by $\Phi_Q=\{\det(zI-Q_q)\}_{q\in\R}$.
For $Q,\widehat{Q}\in\mcM_{\bfA}$, we write $\Phi_Q=\Phi_{\widehat{Q}}$ if $\det(zI-Q_q)=\det(zI-\widehat{Q}_q)$ for all $q\in\R$ and $z\in\C$.
Let $\mcC=L_2(\TMS,\R)$ and
\[
\mcC_{\prigid}=\{f\in\mcC:\Phi_{Q(f)}=\Phi_{Q(g)}\ \mbox{for any}\ g\in[f]_1\},
\]
where the set $[f]_1$ is as in Section \ref{sec:introduction}.
(Notice that, as mentioned in the proof of Lemma \ref{lem:6168}, $\mcE^{(\mu_f)}=\mcE^{(\mu_g)}$ if $\Phi_{Q(f)}=\Phi_{Q(g)}$.)
Unlike Theorem \ref{thm:0426} of a non-rigidity result, the following rigidity result holds:
\begin{equation}
\label{eq:2042}
\mcC_{\prigid}\ \mbox{contains an open set}\ G\ \mbox{of}\ \mcC\ \mbox{such that}\ \Leb(\mcC\setminus G)=0. 
\end{equation}
Using (\ref{eq:6637}), we can prove (\ref{eq:2042}) by an argument similar to that in the proof of \cite[Theorem 1.4]{Nakagawa BBMS 2021}.
\end{remark}



\section*{Funding}
The author is supported by JSPS KAKENHI Grant Number 21K13809.


\end{document}